\documentclass[12pt]{amsproc}

\usepackage{fullpage}
\usepackage{amsmath}
\usepackage{amsfonts}
\usepackage{amssymb}
\usepackage{amsthm}
\usepackage[utf8]{inputenc}
\usepackage{breqn}
\usepackage{afterpage}
\usepackage{longtable}
\usepackage{indentfirst}
\usepackage{caption}
\usepackage{subcaption}
\usepackage{comment} 
\usepackage{enumitem}

\usepackage{graphicx}
\graphicspath{ {./images/} }
\newtheorem{theorem}{Theorem}[section]

\newtheorem{proposition}[theorem]{Proposition}

\theoremstyle{definition}
\newtheorem{definition}[theorem]{Definition}

\newtheorem{theorem-definition}[theorem]{Theorem-Definition}

\theoremstyle{remark}

\numberwithin{equation}{section}

\usepackage{color}
\usepackage{xcolor}

\begin{document}

\title{ $(m, \psi, \delta)-$capacity and its properties}

\begin{author}[K.~Kuldoshev]{Kobiljon Kuldoshev }
    \address{National University of Uzbekistan,  Tashkent, Uzbekistan}
\email{qobiljonmath@gmail.com}
\end{author}

\date{ }
\maketitle 
 { \centering\small \emph{Dedicated to the memory of Professor Azimbay Sadullaev}\par}
 
\begin{abstract}

 In this paper, we introduce the concept of the generalized $(m, \psi, \delta)-$capacity in the complex space $\mathbb{C}^n$, within the class of $m-$subharmonic functions. We give a relation between $(m, \psi, \delta)-$capacity and $(m, \psi, \delta)-$subharmonic measure. Moreover, we prove that the capacity vanishes on $m-$polar sets and vice versa. 
\end{abstract}

\section{Introduction}\label{Intro}
Plurisubharmonic functions (\textit{psh}) and the capacity of the condenser \((K, D)\), defined in terms of them, are fundamental concepts in pluripotential theory. They have attracted significant attention from many researchers due to their theoretical importance and practical applications
(see, for example, \cite{AS}, \cite{BETA}, \cite{KRSS}, \cite{ASKR}, \cite{NAKR}, \cite{KR}). Since the beginning of the 21st century, the class of $m-$subharmonic functions, which properly contains the class of plurisubharmonic functions, has been actively studied (see, for example, \cite{DSKS}, \cite{ZB}). In this class, the concept of condenser capacity was introduced by A. Sadullaev and B. Abdullaev and many important results have been obtained (see \cite{ASBA}).

 Let  $D \subset \mathbb{C}^n$  be an $m-$regular domain (see Section \ref{m-subharmonic} for definition),  $K \subset D$  a fixed compact set, $\psi(z)$ a bounded function on $K $ and $ \delta \in \mathbb{R}$ such that $\delta > \sup\limits_{z \in K} \psi(z).$ In this paper, we introduce the generalized concept of $(m, \psi, \delta)-$capacity, as an extension of the $m-$capacity  which was introduced in \cite{KK}, where $1\le m\le n$. 
 By $sh_m(D)$ we denote the class of $m-$subharmonic functions on $D$.

The function
\[
\omega_m^*(z, K, D, \psi, \delta) = \overline{\lim_{w \to z}} \, \sup \{u(z) \in sh_m(D): u|_K \le \psi|_K,\, u|_D < \delta\}
\]
is called the $(m, \psi, \delta)-$subharmonic measure of the compact set $K$ with respect to $D$. A point $z^0 \in K$ is said to be $(m, \psi, \delta)-$\textit{regular} if
$\omega_m^*(z^0, K, D, \psi, \delta) = \psi(z^0).$
If all the points of a compact set $K$ are  $(m, \psi, \delta)-$regular, then the compact set $K$ is called a $(m, \psi, \delta)-$\textit{regular} compact (see \cite{KKNN}).
The following quantity
\[
C_m(K, D, \psi, \delta) = \inf \left\{ \int_D (dd^c u)^{n - m + 1} \wedge \beta^{m - 1} :
u \in sh_m(D) \cap C(D),\ u|_K \le \psi|_K,\,\ \mathop{\underline{\lim}}_{z \to \partial D} u(z) \ge \delta \right\}
\]
is called the \textit{$(m, \psi, \delta)-$capacity}   of $K$ with respect to $D$, where $d = \partial + \bar{\partial}$, $d^c = \frac{\partial - \bar{\partial}}{4i}$, and $\beta = dd^c |z|^2 = \frac{i}{2} \sum_{i=1}^n dz_i \wedge d\bar{z}_i$ is the standard canonical $(1,1)$ form in $\mathbb{C}^n$.
The $(m, \psi, \delta)-$external capacity, denoted by $C_{m}^{*}(E, D, \psi, \delta)$ for any set $E \subset D$, is defined in the standard way:

\[
C_{m}^{*}(E, D, \psi, \delta) 
= \inf \{ C_{m}(U, D, \psi, \delta) : U \supset E,\; D \supset U-\text{open} \}
\]
where  $
C_m(U, D, \psi, \delta) = \sup \{ C_m(K, D, \psi, \delta) : U \supset K-\text{compact}\}.
$

 Let us now state our main theorem.

\begin{theorem}\label{zero=polar}
Let $ \omega_m^*(z, K, D, \psi, \delta)$ and  $C_m(K, D, \psi, \delta)$ be defined as above. Assume that the function  $\psi(z)$ is continuous on the compact set $K$. Then,  we have
\begin{enumerate}[label=\textnormal{(\alph*)}]
\item \label{A}  if $K$ is  an $(m, \psi, \delta)-$regular compact set,  then the $(m, \psi, \delta)-$measure is maximal in $D \setminus K$, i.e.,
$$(dd^c \, \omega_m^*(z, K, D, \psi, \delta))^{n - m + 1} \wedge \beta^{m - 1} = 0;$$

\item \label{B} if $K$ is an $(m, \psi, \delta)-$regular compact set, then
\[
C_m(K, D, \psi, \delta) = \int_K (dd^c\, \omega_m^*(z, K, D, \psi, \delta))^{n - m + 1} \wedge \beta^{m - 1};
\]

\item \label{C} $ C_{m}^{*}(E, D, \psi, \delta) = 0$ if and only if $E$ is $m-$polar.
\end{enumerate}
\end{theorem}

This paper is organized as follows.   Section \ref{m-subharmonic} is devoted to $m-$subharmonic functions and $m-$capacity.
In Section \ref{Weghted}, we study the weighted  $m-$subharmonic measure and prove the  part \ref{A} of Theorem \ref{zero=polar}.
  Finally, in Section \ref{W cap}, we define $(m, \psi, \delta)-$capacity  and establish some of its properties. Furthermore, we complete the proof of Theorem \ref{zero=polar}.

\section{$m-$subharmonic functions and $m-$capacity} \label{m-subharmonic}
In this section,  we recall the necessary definitions and concepts related to $m-$subharmonic functions from \cite{ASBA}, which will be used throughout this paper. 

The $m-$subharmonic functions in a domain $D \subset \mathbb{C}^n$ are defined using the Hessian operators
\[
(dd^c u)^k \wedge \beta^{n-k}, \quad 1 \le k \le n.
\]
For $k = 1$, the Hessian operator coincides with the Laplace operator and for $k = n$, it coincides with the Monge–Ampère operator.

\begin{definition}\label{DmC}
A twice differentiable function $u(z) \in C^2(D)$ is called \textit{$m-$subharmonic} if the following conditions
\[
(dd^c u)^k \wedge \beta^{n-k} \geq 0, \quad \text{ for all }\, k = 1, 2, \dots, n - m + 1
\]
hold at each point $z^0 \in D$.
\end{definition}
Z. Blocki proved that for all twice differentiable $m-$subharmonic functions $u,\ v_1,\ v_2,$ $\dots,\ v_{n-m}$, the following inequality holds
\begin{equation} \label{BI}
  dd^c u \wedge dd^c v_1 \wedge dd^c v_2 \wedge \dots \wedge dd^c v_{n-m} \wedge \beta^{m-1} \geq 0.  
\end{equation}
On the other hand, if a twice differentiable function $u$ satisfies \eqref{BI} for all twice differentiable $m-$subharmonic functions $v_1,\ v_2,\ \dots,\ v_{n-m}$, then $u$ is necessarily $m-$subharmonic (see \cite{ZB}). Using this, $m-$subharmonic functions are defined in the class of locally integrable functions (see \cite{ASBA}).
\begin{definition}\label{DmL}
A function $u \in L^1_{\text{loc}}(D)$ is called \textit{$m-$subharmonic} in the domain $D \subset \mathbb{C}^n$, if it is upper semicontinuous and for any twice differentiable $m-$subharmonic functions $v_1, v_2, \dots, v_{n-m}$, the current
\[
dd^c u \wedge dd^c v_1 \wedge dd^c v_2 \wedge \dots \wedge dd^c v_{n-m} \wedge \beta^{m-1}
\]
is positive, i.e.,
for any positive test function $\omega$ in $D$, we have
\[
\int u \wedge dd^c v_1 \wedge dd^c v_2 \wedge \dots \wedge dd^c v_{n-m} \wedge \beta^{m-1} \wedge dd^c \omega \geq 0.\]
\end{definition}
Recall that the class of $m-$subharmonic functions is denoted by $sh_m(D)$. It is clear that
\begin{equation} \label{Sub}
psh = sh_1 \subset sh_2 \subset sh_m \subset \dots \subset sh_n = sh.
\end{equation}
\begin{definition}\label{MP}
A set $E \subset D$ is called \textit{$m-$polar} in $D \subset \mathbb{C}^n$ if there exists a function $u \in sh_m(D)$, $u \not\equiv -\infty$, such that $u|_E = -\infty$.
\end{definition}
It follows from \eqref{Sub} that an $m-$polar set is polar. Consequently, for any $m-$polar set $E \subset D$, the Hausdorff measure $H_{2n-2+\varepsilon}(E) = 0$  for any $ \varepsilon>0.$ 
\begin{definition}\label{reg-dom}
A domain $D \subset \mathbb{C}^n$ is called \textit{$m-$regular} if there exists an $m-$subharmonic function $\rho \in sh_m(D)$ such that $\rho|_D < 0$ and 
$\lim\limits_{z \to \partial D} \rho(z) = 0$.
\end{definition}

Let $D \subset \mathbb{C}^n$ be an $m-$regular domain and let $K \subset D$ be a compact set.

\begin{definition}\label{m-cap}
 The following quantity
\[
C_m(K, D) = \inf \left\{ \int_D (dd^c u)^{n - m + 1} \wedge \beta^{m - 1} : u \in sh_m(D) \cap C(D),\, u|_K \le -1,\, \mathop{\underline{\lim}}_{z \to \partial D} u(z) \ge 0 \right\}
\]
is called the\textit{ $m-$capacity } of $K$ with respect to $D$.
\end{definition}

Note that the $m-$capacity of the condenser
$C_m(U, D)$
for an open set $U \subset D$ and the external capacity
$C_m^*(E, D)$
for any set $E \subset D$ are defined in the standard way (see \cite{AS}). The external capacity
$C_m^*(E, D)$
is well studied.
$C_m^*(E, D) = 0$
if and only if $E$ is an $m-$polar set. Moreover, their monotonicity, countable subadditivity and several other properties can be found in the work \cite{ASBA}. 

The comparison principle, which expresses an essential inequality between $m-$subharmonic functions and the integrals of their Hessians, is used extensively throughout this paper. For this reason, we recall its statement below. 

Let $u, v \in sh_m(D) \cap L_{loc}^\infty(D)$.  If
$F = \{ z \in D : u(z) < v(z) \} \subset\subset D,$
then the following inequalities hold
\begin{equation}\label{ICP}
\int_F (dd^c u)^k \wedge \beta^{n-k} \ge \int_F (dd^c v)^k \wedge \beta^{n-k}, \quad 1 \le k \le n - m + 1.
\end{equation}
    
\section{Weighted $m-$subharmonic measure}\label{Weghted}
In this section, we review some important notions and properties of the weighted $m-$subharmonic measure, then prove part \ref{A} of Theorem \ref{zero=polar}.  Let $D \subset \mathbb{C}^n$ be an $m-$regular domain, $E \subset D$ be any fixed set and $\psi(z)$ be a bounded function in $E$. We denote by
$\mathcal{U}(E, D, \psi, \delta)$
the class of all functions $u(z) \in sh_m(D)$ such that
\[
u|_E \le \psi|_E \quad \text{and} \quad u|_D < \delta,
\]
where $\delta \in \mathbb{R}$. We define the function
\[
\omega_m(z, E, D, \psi, \delta) = \sup \{ u(z) : u(z) \in \mathcal{U}(E, D, \psi, \delta) \}.
\]
Recall that the function
\[
\omega_m^*(z, E, D, \psi, \delta) = \overline{\lim_{w \to z}} \, \omega_m(w, E, D, \psi, \delta)
\]
is called the \textit{$(m, \psi, \delta)-$subharmonic measure} (\textit{$\mathcal{P}_{(m, \psi, \delta)}-$measure}) of the set $E$ with respect to $D$.

Note that $\omega^*(z, E, D, -1, 0), \quad \psi \equiv -1, \quad \delta = 0,$
coincides with the $m-$subharmonic measure of the potential theory in the class of functions $u(z) \in sh_m(D)$, i.e.,
\[
\omega_m^*(z, E, D, -1, 0) = \omega_m^*(z, E, D).
\]
By definition, the function $\omega_m^*(z, E, D, \psi, \delta)$ is $m-$subharmonic in $D$ and the inequality
\[
\omega_m^*(z, E, D, \psi, \delta) \le \delta
\]
holds for all $z \in D$. Moreover,   if $E \Subset D$, then $\displaystyle \lim_{z \to \partial D} \omega_m^*(z, E, D, \psi, \delta) = \delta$. Furthermore,
The weighted $m-$subharmonic measure satisfies the properties of the unweighted $m-$subharmonic measure (see \cite{KKNN}, \cite{KKKR}).

Let the function $\psi(z)$ be extended to the domain $D$, as a function from the class $\mathcal{U}(E, D, \psi, \delta)$, i.e., if there is a function $\tilde{\psi} \in sh_m(D)$ such that
\[
 \tilde{\psi}|_E = \psi|_E \quad \text{and} \quad \tilde{\psi}|_D < \delta
\]
then it is obvious that
\[
\omega_m(z, E, D, \psi, \delta) \ge \tilde{\psi}(z), \quad \forall z \in D
\]
and
\begin{equation}\label{e}
\omega_m(z, E, D, \psi, \delta) = \psi(z), \quad \forall z \in E.
\end{equation}
However, in general, equality \eqref{e} does not hold (see \cite{KKNN}). 
In this paper, we consider the special case where
$\delta > \sup\limits_{z \in E} \psi(z)$
and condition \eqref{e} is satisfied in the definition of
$\omega_m(z, E, D, \psi, \delta).$
Below, we present an important theorem concerning the continuity of the $(m, \psi, \delta)-$subharmonic measure, which was proved in our work \cite{KKNN}.
\begin{theorem}\label{con}
Let $K$ be an $(m, \psi, \delta)-$ regular compact set and let $\psi(z)$ be continuous on the compact set $K$. Then  $\omega_m(z, K, D, \psi, \delta) \in C(\bar{D})$.

\end{theorem}
\subsection{The maximality of the weighted $m-$subharmonic measure.}
Let $D$ be a bounded domain in $\mathbb{C}^n$.
\begin{definition}
A function $u \in sh_m(D)$ is called \textit{maximal} in the domain $D$ if it satisfies the dominance principle within the class of $m-$subharmonic functions, that is, if  
\[
\forall v \in sh_m(D): \quad \mathop{\underline{\lim}}_{z \to \partial D} \left( u(z) - v(z) \right) \ge 0,
\]
then for all $z \in D$, the inequality $u(z) \ge v(z)$ holds.
\end{definition}
We will prove that the $(m, \psi, \delta)-$subharmonic measure is maximal in $D \setminus K $, which corresponds to the first part of Theorem~\ref{zero=polar}.
\begin{proposition}\label{m-th}
Let $K \subset D$ be an $(m, \psi, \delta)-$regular compact set and let $\psi(z)$ be continuous on the compact set $K$. Then, the $\mathcal{P}_{(m, \psi, \delta)}-$measure is maximal in the open set $D \setminus K$, i.e.,
\[
(dd^c \, \omega_m^*(z, K, D, \psi, \delta))^{n - m + 1} \wedge \beta^{m - 1} = 0.
\]
\end{proposition}
\begin{proof}
According to Theorem \ref{con},  for any $z \in D$, we have
\[
\omega_m(z, K, D, \psi, \delta) \equiv \omega_m^*(z, K, D, \psi, \delta) \in C(\bar{D}).
\]
 We fix a ball $B \subset\subset D \setminus K$ and construct the following function
\[
v(z) = \sup \left\{ u(z) : u \in sh_m(B) \cap C(\bar{B}),\ u|_{\partial B} \le \omega^*(z, K, D, \psi, \delta)|_{\partial B} \right\}.
\]
Then, by \cite{ZB}, we have $v \in sh_m(B) \cap C(\bar{B})$ and it is maximal in the ball $B$, i.e.,
\[
(dd^c v)^{n - m + 1} \wedge \beta^{m - 1} = 0,
\]
and
\[
v(z) = \omega_m^*(z, K, D, \psi, \delta) \quad \text{for all } z \in \partial B.
\]
Since $v(z)$ is maximal in $B$, it follows that 
\[
v(z) \ge \omega_m^*(z, K, D, \psi, \delta),
\]
 for all $ z \in B$. Let us define the following function 
\[
w(z) =
\begin{cases}
\omega_m^*(z, K, D, \psi, \delta), & z \in D \setminus B, \\
v(z), & z \in B.
\end{cases}
\]
 It is not difficult to  show (see \cite[page 162]{ASBA}) that,  by its definition, we have
\[
w(z) \in sh_m(D) \cap C(D) \quad \text{and} \quad w(z) \in \mathcal{U}(K, D, \psi, \delta).
\]
As a result,
\[
w(z) \le \omega_m^*(z, K, D, \psi, \delta), \quad \forall z \in D.
\]
Consequently, for all $z\in B$, 

$$v(z) = \omega_m^*(z, K, D, \psi, \delta).$$
From the arbitrariness of the ball $B \subset D \setminus K$, we can conclude that
\[
(dd^c\, \omega_m^*(z, K, D, \psi, \delta))^{n - m + 1} \wedge \beta^{m - 1} = 0
\]
in the open set $D \setminus K$.
\end{proof}

\section{$(m, \psi, \delta)-$capacity}\label{W cap}
This section is devoted to the study of the $(m, \psi, \delta)-$capacity introduced in Section \ref{Intro} and to the proof of parts \ref{B} and \ref{C} of our main result. We recall the following definition, which will play a central role in what follows. Let $K$ be a compact subset of an $m-$regular domain $D \subset \mathbb{C}^n$ and let $\psi(z)$ be a function defined on $K$.  Fix also $\delta$ with $\delta> \sup\limits_{z\in K} \psi(z)$.
\begin{definition}\label{weighted m-cap}
 The following quantity
\[
C_m(K, D, \psi, \delta) = \inf \left\{ \int_D (dd^c u)^{n - m + 1} \wedge \beta^{m - 1} :
u \in sh_m(D) \cap C(D),\ u|_K \le \psi|_K,\,\ \mathop{\underline{\lim}}_{z \to \partial D} u(z) \ge \delta \right\}
\]
is called the \textit{$(m, \psi, \delta)-$capacity} of $K$ with respect to $D$.  
\end{definition}

Note that $C_m(K, D, -1, 0)$, where $\psi = -1$ and $\delta = 0$, coincides with the classical $m-$capacity of $K$ with respect to $D$ in potential theory within the class of $u \in sh_m(D)$, i.e,
\[
C_m(K, D, -1, 0) = C_m(K, D).
\]
The weighted $(m, \psi)-$capacity $C_m(K, D, \psi)$, in the case $\delta = 0$, was studied in our previous work \cite{KK}.  Throughout this paper, $\psi(z)$ is assumed to be continuous on $K$ and to satisfy $\psi|_K<\delta $ whenever the $(m, \psi, \delta)-$capacity $C_m(K, D, \psi, \delta)$ is considered. 
The $(m, \psi, \delta)-$capacity has the following properties.
\begin{proposition}\label{monotonicity}
The following monotonicity properties hold:
\begin{enumerate}[label=\textnormal{(\alph*)}]
\item if $K_1 \subset K_2 \subset D$, then
$C_m(K_1, D, \psi, \delta) \le C_m(K_2, D, \psi, \delta)$;
\item if $\psi_1 \le \psi_2$, then $C_m(K, D, \psi_1, \delta) \ge C_m(K, D, \psi_2, \delta)$;
\item if $\delta_1 \le \delta_2$, then $C_m(K, D, \psi, \delta_1) \le C_m(K, D, \psi, \delta_2).$
\end{enumerate}
\end{proposition}
The proof of the monotonicity properties of the $(m, \psi, \delta)-$capacity $C_m(K, D, \psi, \delta)$ follows easily from its definition. We now prove the following proposition, which corresponds to part \ref{B} of Theorem \ref{zero=polar}.
\begin{proposition}\label{reg-com cap}
If $K \subset D$ is an $(m, \psi, \delta)-$regular compact set, then
\[
C_m(K, D, \psi, \delta) = \int_K (dd^c\, \omega_m^*(z, K, D, \psi, \delta))^{n - m + 1} \wedge \beta^{m - 1}.
\]
\end{proposition}
\begin{proof}
By Theorem \ref{con}, we have     
\[
\omega_m^*(z, K, D, \psi, \delta) \in sh_m(D) \cap C(D).
\]
Furthermore, according to Proposition \ref{m-th} and Definition \ref{weighted m-cap}, the following inequality holds
\[
\int_K (dd^c\, \omega_m^*(z, K, D, \psi, \delta))^{n - m + 1} \wedge \beta^{m - 1}
\ge C_m(K, D, \psi, \delta).
\]

On the other hand, fix an arbitrary $\varepsilon > 0$ satisfying the inequality
$0 < 2\varepsilon < \delta - \max\limits_{z \in K} \psi(z).$
Then, for any function $u \in sh_m(D) \cap C(D)$ such that
\[
u|_K \le \psi|_K \quad \text{and} \quad \mathop{\underline{\lim}}_{z \to \partial D} u(z) \ge \delta,
\]
the set
\[
F = \left\{ z \in D :\ 
u(z) < \left(1 - \frac{2\varepsilon}{\delta - \max\limits_{z \in K} \psi(z)}\right) \cdot \omega_m^*(z, K, D, \psi, \delta) 
+ \varepsilon \cdot \frac{\delta+\max\limits_{z \in K} \psi(z)}{\delta - \max\limits_{z \in K} \psi(z)} \right\}
\]
is open and it is easy to verify that  $K \subset F \Subset D$.  
Therefore, according to the comparison principle, we have
\[
\int_F (dd^c u)^{n - m + 1} \wedge \beta^{m - 1}
\ge
\left( 1 - \frac{2\varepsilon}{\delta - \max\limits_{z \in K} \psi(z)} \right)^{n - m + 1}
\cdot
\int_F (dd^c\, \omega_m^*(z, K, D, \psi, \delta))^{n - m + 1} \wedge \beta^{m - 1}.
\]
By Proposition \ref{m-th} and since \( K \subset F \), we obtain 
\[
\left( 1 - \frac{2\varepsilon}{\delta - \max\limits_{z \in K} \psi(z)} \right)^{n - m + 1}
\cdot \int_K (dd^c\, \omega_m^*(z, K, D, \psi, \delta))^{n - m + 1} \wedge \beta^{m - 1}
\]
\[
\le \int_F (dd^c\, u)^{n - m + 1} \wedge \beta^{m - 1}
\le \int_D (dd^c\, u)^{n - m + 1} \wedge \beta^{m - 1}.
\]
The arbitrariness of \( \varepsilon > 0 \) implies that
\[
\int_K (dd^c\, \omega_m^*(z, K, D, \psi, \delta))^{n - m + 1} \wedge \beta^{m - 1}
\le C_m(K, D, \psi, \delta).
\]
\end{proof}
\begin{proposition}\label{com-cap}
For any compact $K \subset D,$
\[
C_m(K, D, \psi, \delta) = \inf \left\{ C_m(E, D, \tilde{\psi}, \delta) :\ E \supset K \right\},
\]
where $\tilde{\psi} \in C(E)$, $\tilde{\psi}|_K = \psi|_K$ and $E$ is an $(m, \tilde{\psi}, \delta)-$regular compact set in the domain $D$.
\end{proposition}
\begin{proof}
From the definition of $C_m(K, D, \psi, \delta)$, for any $\varepsilon > 0$, there exists a function $u \in sh_m(D) \cap C(D)$ such that
$$u|_K \le \psi|_K \quad \text{and} \quad \mathop{\underline{\lim}}_{z \to \partial D} u(z) \ge \delta$$ and the following inequality holds
\[
\int_D (dd^c u)^{n - m + 1} \wedge \beta^{m - 1} - C_m(K, D, \psi, \delta) < \varepsilon.
\]
 Since the function $\psi(z)$ is continuous on the compact set $K$, according to Whitney's theorem \cite{WH}, there exists some continuous function $\tilde{\psi}(z)$ in $D$ such that $\tilde{\psi}|_K = \psi|_K$. Then the set 
\[
U = \{ z \in D : u(z) < \tilde{\psi}(z) + \varepsilon \}
\]
is open, contains the compact set $K$. 
Let us consider the open set
\[
F = \left\{ z \in D : u(z) < \left( 1 - \frac{2\varepsilon}{\delta - \max\limits_{z \in E} \tilde{\psi}(z)} \right) \cdot \omega_m^*(z, E, D, \tilde{\psi}, \delta) + \varepsilon \cdot \frac{\delta+\max\limits_{z \in E} \psi(z)}{\delta - \max\limits_{z \in E} \psi(z)} \right\},
\]
where $E$ is an $(m, \tilde{\psi}, \delta)-$regular compact set such that $K \subset E \subset U$ and   $0 < 2\varepsilon < \delta - \max\limits_{z \in E} \tilde{\psi}(z)$.
It is not difficult to check that $E \subset F \subset\subset D.$  
By Proposition \ref{reg-com cap} and the comparison principle, we have
\begin{align*}
 C_m(E, D, \tilde{\psi}, \delta)&
= \int_E (dd^c \omega_m^*(z, E, D, \tilde{\psi}, \delta))^{n - m + 1} \wedge \beta^{m - 1}\\
&= \int_F (dd^c \omega_m^*(z, E, D, \tilde{\psi}, \delta))^{n - m + 1} \wedge \beta^{m - 1}\\
&\le \frac{1}{\left( 1 - \frac{2\varepsilon}{\delta - \max\limits_{z \in E} \tilde{\psi}(z)} \right)^{n - m + 1}} \cdot
\int_F (dd^c u)^{n - m + 1} \wedge \beta^{m - 1}\\
&\le \frac{1}{\left( 1 - \frac{2\varepsilon}{\delta - \max\limits_{z \in E} \tilde{\psi}(z)} \right)^{n - m + 1}}
\int_D (dd^c u)^{n - m + 1} \wedge \beta^{m - 1}\\
&
< \frac{1}{\left( 1 - \frac{2\varepsilon}{\delta - \max\limits_{z \in E} \tilde{\psi}(z)} \right)^{n - m + 1}}
\left( C_m(K, D, \psi, \delta) + \varepsilon \right).
\end{align*}
Thus,
\[C_m(K, D, \psi, \delta) \le C_m(E, D, \tilde{\psi}, \delta)
< \frac{1}{\left( 1 - \frac{2\varepsilon}{\delta - \max\limits_{z \in E} \tilde{\psi}(z)} \right)^{n - m + 1}}
\left( C_m(K, D, \psi, \delta) + \varepsilon \right).\]
The arbitrariness of $ \varepsilon > 0 $ implies that 
\[C_m(K, D, \psi, \delta) = \inf \left\{ C_m(E, D, \tilde{\psi}, \delta) : E \supset K \right\}.\]
\end{proof}

\begin{proposition}\label{ext-cap-com-eq}
For any compact set $K \subset D$, the following equality holds
\[
C_{m}^{*}(K,D,\psi,\delta) = C_{m}(K,D,\psi,\delta).
\]    
\end{proposition}
\begin{proof}
 From the definition of the $(m,\psi,\delta)-$external capacity, it follows that
\[
C_{m}^{*}(K,D,\psi,\delta)\ge C_{m}(K,D,\psi,\delta).
\]
 On the other hand, if $K$ is an $(m,\psi,\delta)-$regular compact set, then for a sufficiently small fixed $\varepsilon > 0$, we construct the open set
\[
U = \{ z \in D : \omega_{m}^{*}(z,K,D,\psi,\delta) < \tilde{\psi}(z) + \varepsilon^2 \},
\]
where $\tilde{\psi}(z) \in C(D)$ satisfies $\tilde{\psi}|_{K} = \psi|_{K}$ and $\tilde{\psi}(z) + 2\varepsilon < \delta$ for any $z \in D$.
 It is easy to verify that $K \subset U$. In fact, for any $(m,\tilde{\psi},\delta)-$regular compact set $F$ such that $K \subset F \subset U$ and for any $z \in D$, the set
\[
G = \left\{ z \in D : \frac{\omega_{m}^{*}(z,K,D,\psi,\delta) - \delta}{1 - \varepsilon} + \varepsilon^2 < \omega_{m}^{*}(z,F,D,\tilde{\psi},\delta) - \delta \right\}
\]
is open and satisfies $F \subset G \Subset D$.

 Thus,
\[
\frac{1}{(1 - \varepsilon)^{n-m+1}} \int_{G} (dd^c \omega_m^{*}(z,K,D,\psi,\delta))^{n-m+1}\wedge \beta^{m-1}
\ge \int_{G}(dd^c \omega_m^{*}(z,F,D,\psi,\delta))^{n-m+1}\wedge \beta^{m-1}.
\]
By Proposition \ref{m-th}, we have the inequality
\[
\frac{1}{(1 - \varepsilon)^{n-m+1}} \int_{K}(dd^c \omega_m^{*}(z,K,D,\psi,\delta))^{n-m+1}\wedge \beta^{m-1}
\ge \int_{F}(dd^c \omega_m^{*}(z,F,D,\tilde{\psi},\delta))^{n-m+1}\wedge \beta^{m-1}.
\]
Therefore,
\[
C_{m}^{*}(K,D,\psi,\delta)\le C_{m}(U,D,\psi,\delta)
\le \frac{1}{(1 - \varepsilon)^{n-m+1}}\int_{K}(dd^c \omega_m^{*}(z,K,D,\psi,\delta))^{n-m+1}\wedge \beta^{m-1}.
\]
Since $\varepsilon>0$ is arbitrary and by Proposition \ref{reg-com cap}, we obtain
\[
C_{m}^{*}(K,D,\psi,\delta)\le C_{m}(K,D,\psi,\delta).
\]
Now let $K\subset D$ be  compact but not  $(m,\psi,\delta)-$ regular. Then, by monotonicity, for any $(m,\tilde{\psi},\delta)-$regular compact set $E\supset K$, where $\tilde{\psi} \in C(D)$ such that $\tilde{\psi}|_K=\psi|_K$, we have
\[
C_{m}^{*}(K,D,\psi,\delta)\le C_{m}^{*}(E,D,\tilde{\psi},\delta)= C_{m}(E,D,\tilde{\psi},\delta).
\]
Consequently, since $E$ is an arbitrary $(m,\tilde{\psi},\delta)-$regular compact set and by Proposition \ref{com-cap}, it follows that
\[
C_{m}^{*}(K,D,\psi,\delta)\le C_{m}(K,D,\psi,\delta).
\]
\end{proof}
The conclusion of Proposition \ref{ext-cap-com-eq} is one of the essential steps in the proof of part \ref{C} of our main result and we are now ready to prove it.

\begin{proof}[{End of the proof of Theorem \ref{zero=polar}}]
In \cite{ASBA}, it is proved that the unweighted $m-$external capacity $ C_{m}^{*}(E, D) = 0$  if and only if $E$ is $m-$polar. Thus, to prove part \ref{C} of Theorem \ref{zero=polar}, it suffices to show that for any $(m, \psi, \delta) -$regular compact set $E$, there exist positive constants  $C_1$  and $C_2$ such that  
\[
C_1 \cdot C_{m}^{*}(E, D)
\leq C_{m}^{*}(E, D, \psi, \delta)
\leq C_2 \cdot C_{m}^{*}(E, D).
\]

Let us assume that $E$ is an $(m, \psi, \delta)-$regular compact set. Since $D \subset \mathbb{C}^{n}$ is an $m-$regular domain, there exists a function $\rho \in sh_m(D)$ such that $\rho|_D < 0$ and $\lim\limits_{z \to \partial D} \rho(z) = 0$. According to Proposition 1.3 in \cite{KKNN}, for any $z \in D$, we have
\[
\left( \delta - \min_{z \in E} \psi(z) \right) \cdot \omega_m^{*}(z, E, D) + \delta + \rho(z)
< \omega_m^{*}(z, E, D, \psi, \delta).
\]
Fix an arbitrary $\varepsilon > 0$ such that
$\varepsilon < -\max\limits_{z \in E} \rho(z)$
and define the open set
\[
U = \left\{ z \in D : \left( \delta - \min_{z \in E} \psi(z) \right) \omega_m^{*}(z, E, D) + \delta + \varepsilon \rho(z) + \varepsilon^2 < \omega_m^{*}(z, E, D, \psi, \delta) \right\}.
\]
It is easy to verify that $ E \Subset U \Subset D$. According to the comparison principle, the following inequality holds
\[
 \int_U (dd^c \omega_m^{*}(z, E, D, \psi, \delta))^{n - m + 1} \wedge \beta^{m - 1}
\le \int_U \left[ dd^c\left( \left( \delta - \min_{z \in E} \psi(z) \right) \omega_m^{*}(z, E, D) + \varepsilon \rho(z) \right) \right]^{n - m + 1} \wedge \beta^{m - 1} 
\]

It is known that the integral
\[
\int_U \left[ dd^c\left( \left( \delta - \min_{z \in E} \psi(z) \right) \omega_m^{*}(z, E, D) + \varepsilon \rho(z) \right) \right]^{n - m + 1} \wedge \beta^{m - 1}
\]
can be expressed as
\[
\left( \delta - \min_{z \in E} \psi(z) \right)^{n - m + 1} \int_U (dd^c \omega_m^{*}(z, E, D))^{n - m + 1} \wedge \beta^{m - 1} + \varepsilon C,
\]
where $C$ is a positive constant. According to Proposition \ref{m-th},
\[
\int_{E} (dd^c \omega_m^{*}(z, E, D, \psi, \delta))^{n - m + 1} \wedge \beta^{m - 1}\le\left( \delta - \min_{z \in E} \psi(z) \right)^{n - m + 1} \int_{E} (dd^c \omega_m^{*}(z, E, D))^{n - m + 1} \wedge \beta^{m - 1} + \varepsilon C
\]
It is known that 
$$C_{m}^{*}(K, D)=\int_{K} (dd^c \omega_m^{*}(z, K, D))^{n - m + 1} \wedge \beta^{m - 1}$$ for any compact $K \subset D$ (see \cite{AS}). Consequently,
by Proposition \ref{ext-cap-com-eq} and from the arbitrariness of $\varepsilon > 0$, we have
\[
C_{m}^{*}(E, D, \psi, \delta) \le \left( \delta - \min_{z \in E} \psi(z) \right)^{n - m + 1} C_{m}^{*}(E, D).
\]

On the other hand, applying Proposition 1.3 from \cite{KKNN} again, we obtain that for any \( z \in D \),
\[
\omega_m^{*}(z, E, D, \psi, \delta) + \rho(z)
< \left( \delta - \max_{z \in E} \psi(z) \right) \cdot \omega_m^{*}(z, E, D)+\delta.
\]
Fix an arbitrary $\varepsilon > 0$ such that
$\varepsilon < -\max\limits_{z \in E} \rho(z)$
and we define the set
\[
F = \left\{ z \in D : \omega_m^{*}(z, E, D, \psi, \delta) + \varepsilon \rho(z) + \varepsilon^2 
< \left( \delta - \max_{z \in E} \psi(z) \right) \cdot \omega_m^{*}(z, E, D)+\delta \right\}.
\]
It is easy to check that $ E \subset F \Subset D$ and  by the comparison principle, we get 
\[
\int_F [ (\delta - \max_{z \in E} \psi(z) )dd^c(\omega_m^{*}(z, E, D))]^{n - m + 1} \wedge \beta^{m - 1} 
\le \int_F \left[dd^c (\omega_m^{*}(z, E, D, \psi, \delta)+\varepsilon \rho(z))\right]^{n - m + 1} \wedge \beta^{m - 1}.
\]
Using the above technique again and since 
$\varepsilon>0$ is arbitrary, we obtain the following inequality 
$$ [\delta - \max_{z \in E} \psi(z)]^{n - m + 1} \, C_{m}^{*}(E, D)\le C_{m}^{*}(E, D, \psi, \delta).$$   
\end{proof}

\textbf{Acknowledgments.}
  The author is grateful to the late Professor Azimbay Sadullaev and Professor Karim Rakhimov for their insightful comments and valuable guidance during this work. This work is dedicated to the memory of Professor Sadullaev.


\begin{thebibliography}{99}
 
\bibitem{ASBA} A. Sadullaev, B. Abdullaev, \emph{Potential theory in the class of $m-$subharmonic functions.} Trudy Mat. Inst. Steklova, 279, 2012, pp. 166-192 (in Russian).
\bibitem{AS} A. Sadullaev, \emph{Pluripotential theory. Applications.}  Palmarium Academic Publishing,  2012, 307 p. (in Russian).
\bibitem{BETA} E. Bedford, B. A. Taylor, \emph{A new capacity for plurisubharmonic functions.} Acta Math, 149, no 1-2, 1982, pp. 1-40.
\bibitem{DSKS} S. Dinew, S. Kolodziej, \emph{A priori estimates for the complex Hessian equation.} Analysis and PDE, Vol. 7, Issue 5, 2014, pp. 227–244.

\bibitem{ZB} Z. Blocki, \emph {Weak solutions to the complex Hessian equation.} Ann. Inst. Fourier, Grenoble, Vol. 5, №. 55, 2005, pp.1735-1756.

\bibitem{KRSS} K. Rakhimov, Sh. Shopulatov, \emph{A mean value criterion for plurisubharmonic functions.} Complex Variables and Elliptic Equations, Vol. 67, Issue 12, 2021, pp. 2866-2877.

\bibitem{KK} K. Kuldoshev, \emph{Weighted $(m,\psi)-$ capacity $C_m(K,D,\psi)$ of a condenser $(K,D)$.} Uzbek Mathematical Journal, Vol. 69, Issue 2, 2025, pp. 153-164.
\bibitem{ASKR} A. Sadullaev, K. Rakhimov,  \emph{Capacity dimension of the Brjuno set.} Indiana University Mathematics Journal, Vol. 64, Issue 6, 2015, pp. 1829-1834.
\bibitem{NAKR}  N. Akramov, K. Rakhimov, \emph{Capacity dimension of the Brjuno set in ${\mathbb {C}}^ n $.} Preprint arXiv:2505.18801, 2025.
\bibitem{KR}  K. Rakhimov, \emph{Capacity dimension of Perez-Marco set.} Contemporary Mathematics, Volume 662, 2016,
pp. 131-138.
\bibitem{KKNN} K. Kuldoshev, N. Narzullayev, \emph{$(m,\psi,\delta)-$regularity of compacts in ${{\mathbb{C}}^{n}}$.} Acta NUUz Vol 2.2.1, 2024, pp. 65-76.
\bibitem{KKKR}  K. Kuldoshev, K. Rakhimov, \emph{On the Hölder continuity of weighted $m-$extremal functions.} Accepted for publication in Journal of Siberian Federal University. Mathematics and Physics.
\bibitem{WH} H. Whitney, \emph{Analytic extensions of differentiable functions defined in closed sets}. Trans. American Mathematical Society, Vol.36, 1934, pp.63-89.
     \end{thebibliography}
\end{document}